\documentclass[11pt]{amsart}
\usepackage{amsfonts}

\newtheorem{thm}{Theorem}[section]

\newtheorem{prop}[thm]{Proposition}

\theoremstyle{definition}
\newtheorem{exmp}[thm]{Example}

\theoremstyle{remark}
\newtheorem{rem}[thm]{Remark}

\newcommand{\ulambda}{{\boldsymbol{\lambda}}}

\begin{document}
\title{On the Mullineux involution for Ariki-Koike algebras}
\author{Nicolas Jacon and C\'edric Lecouvey}
\address{N.J.:Universit\'e de Franche-Comt\'e, UFR Sciences et Techniques, 16 route
de Gray, 25 030 Besan\c{c}on, France.}
\email{njacon@univ-fcomte.fr}
\address{C. L. Laboratoire de Math\'{e}matiques Pures et Appliqu\'{e}es Joseph
Liouville Centre Universitaire de la Mi-Voix B.P. 699 62228 Calais France}
\email{Cedric.Lecouvey@lmpa.univ-littoral.fr}
\date{April, 2008}

\begin{abstract}
This note is concerned with a natural generalization of the Mullineux
involution for Ariki-Koike algebras. By using a result of Fayers together
with previous results by the authors, we give an efficient algorithm for
computing this generalized Mullineux involution. Our algorithm notably does
not involve the determination of paths in affine crystals.
\end{abstract}

\maketitle


\pagestyle{myheadings}

\markboth{Nicolas Jacon and C\'edric Lecouvey}{Mullineux involution for
Ariki-Koike algebras}


\section{Introduction}

It is known that the set of simple modules of the symmetric group over a
field of characteristic $p$ can be labelled by the set $\mathcal{P}_{p}$ of $%
p$-regular partitions, that is, the partitions where no non zero part is
repeated $p$ or more times. The tensor product of one of these simple
modules with the sign representation is again simple. This gives a natural
involution $I$ on $\mathcal{P}_{p}.$ In 1979, Mullineux \cite{Mu} introduced
a combinatorial algorithm yielding an involution $m$ on $\mathcal{P}_{p}$
and conjectured the equality $m=I$.\ In 1994, Kleshchev \cite{KL} obtained
the first combinatorial description of the involution $I$ by using sequences
of $p$-regular partitions.\ In particular, the corresponding algorithm $k$
is different from the combinatorial procedure $m$ of Mullineux.\ Finally,
the conjecture $m=I$ (and thus the equality $m=k$) was proved by Ford and
Kleshchev in \cite{KF} (see also \cite{BO} and \cite{BK} for different
proofs). For the sake of simplicity, we will refer to the involution I=m as
the Mullineux involution in the sequel.

A natural quantum analogue of this problem has been considered by Richards 
\cite{Rich} and by Brundan \cite{B} in the context of the Hecke algebra of
type $A$ (see also \cite{BK} and \cite{LLT}).\ The Mullineux involution can
then be interpreted as a skew-isomorphism of $A_{e-1}^{(1)}$-crystals, that
is as an isomorphism of oriented graphs switching the sign of each arrow.
The resulting algorithm then coincides with that introduced by Kleshchev in 
\cite{KL}.

In this note, we study an analogue of the Mullineux involution for the
Ariki-Koike algebras (also called cyclotomic Hecke algebras of type $%
G(l,1,n) $ in the literature). Let $l,n$ be positive integers, $R$ be a
field of arbitrary characteristic and consider $(q,Q_{0},Q_{1},...,Q_{l-1})$
an $l+1$-tuple of invertible elements in $R$ such that $q\neq 1$. The
Ariki-Koike algebra $\mathcal{H}_{R,n}:=\mathcal{H}_{R,n}(q;Q_{0},$ $%
Q_{1},...,Q_{l-1})$ over $R$ is the unital associative $R$-algebra presented
by:

\begin{itemize}
\item  generators: $T_0$, $T_1$,..., $T_{n-1}$,

\item  relations: 
\begin{align*}
& T_0 T_1 T_0 T_1=T_1 T_0 T_1 T_0, \\
& T_iT_{i+1}T_i=T_{i+1}T_i T_{i+1}\ (i=1,...,n-2), \\
& T_i T_j =T_j T_i\ (|j-i|>1), \\
&(T_0-Q_0)(T_0-Q_1)...(T_0- Q_{l-1}) = 0, \\
&(T_i-q)(T_i+1) = 0\ (i=1,...,n-1).
\end{align*}
\end{itemize}

If $l=1$ (resp. $l=2$), we obtain a Hecke algebra of type $A$ (resp. of type 
$B$). To study these algebras, by a theorem of Dipper and Mathas \cite{DM},
it is sufficient to consider the case where $Q_{0}=q^{s_{0}}$, $%
Q_{1}=q^{s_{1}}$, ..., $Q_{l-1}=q^{s_{l-1}}$ for ${\mathbf{s}}%
=(s_{0},s_{1},...,s_{l-1})$ an $l$-tuple of integers. The algebra $\mathcal{H%
}_{R,n}$ is then a cellular algebra in the sense of Graham and Lehrer \cite
{GL}. As a consequence, we can define Specht modules which are indexed by
the $l$-partitions of rank $n$. Recall that an $l$-partition ${{%
\boldsymbol{\lambda}}}$ of rank $n$ is a sequence of $l$ partitions ${%
\boldsymbol{\lambda}}=(\lambda ^{(0)},...,\lambda ^{(l-1)})$ such that $%
\displaystyle{\sum_{k=0}^{l-1}{|\lambda ^{(k)}|}}=n$. We denote by $\Pi
_{l,n}$ the set of $l$-partitions of rank $n$. Hence, to each ${%
\boldsymbol{\lambda}}\in \Pi _{l,n}$ is associated a Specht module $S^{{%
\boldsymbol{\lambda}}}$. Let $e$ be the multiplicative order of $q$ (we put $%
e=\infty $ if $q$ is not a root of unity). The Specht modules are in general
not irreducible. Nevertheless, there exists a natural bilinear form, $%
\mathcal{H}_{R,n}$-invariant, on each of these modules and an associated
radical such that the quotients $D^{{\boldsymbol{\lambda}}}:=S^{{%
\boldsymbol{\lambda}}}/\text{rad}(S^{{\boldsymbol{\lambda}}})$ are $0$ or
irreducible. Moreover, the non-zero $D^{{\boldsymbol{\lambda}}}$ provide a
complete set of non isomorphic simple modules. Let $\Phi _{e}^{\frak{s}}(n)$
be the set of $l$-partitions such that $D^{{\boldsymbol{\lambda}}}$ is non
zero (where we set $\mathfrak{s}=(s_{0}(\text{mod }e),...,s_{l-1}(\text{mod }%
e))$). It has been shown by Ariki \cite{A} and Ariki-Mathas \cite{AM}, that
this set equals to the set of Kleshchev $l$-partitions. These $l$-partitions
are generalizations of the $e$-regular partitions.\ They appear as the
vertices of a distinguished realization of the abstract $A_{e-1}^{(1)}$%
-crystal $B_{e}(\Lambda _{\mathfrak{s}})$ associated to the irreducible ${%
\mathcal{U}_{v}(\widehat{\mathfrak{sl}_{e}})}$-module of highest weight $%
\Lambda _{\mathfrak{s}}$ (see section \ref{subsec_realization}).


Following \cite{Fayers}, denote by $\widetilde{\mathcal{H}}_{R,n}$ the
algebra $\mathcal{H}_{R,n}(q^{-1};s_{l-1},...,s_{0})$ and write $\widetilde{T%
}_{0}$,...,$\widetilde{T}_{l-1}$ for the standard generators of $\widetilde{%
\mathcal{H}}_{R,n}$. If ${\boldsymbol{\lambda}}\in \Pi _{l,n}$, we write $%
\widetilde{S}^{{\boldsymbol{\lambda}}}$ for the corresponding Specht module
on $\widetilde{\mathcal{H}}_{R,n}$ and $\widetilde{D}^{{\boldsymbol{\lambda}}%
}$ for the corresponding irreducible one. We have an isomorphism $\theta :%
\mathcal{H}_{R,n}\rightarrow \widetilde{\mathcal{H}}_{R,n}$ given by 
\begin{equation*}
T_{0}\mapsto \widetilde{T}_{0}\qquad T_{i}\mapsto -q\widetilde{T}_{i}\
(i=1,...,n-1).
\end{equation*}
Put $\widetilde{\mathfrak{s}}=(-s_{l-1}($mod $e),...,-s_{0}($mod $e))$.
Then, $\theta $ induces a functor $F_{l}$ from the category of $\widetilde{%
\mathcal{H}}_{R,n}$-modules to the category of ${\mathcal{H}}_{R,n}$
modules. As a consequence, we obtain a bijective map 
\begin{equation*}
m_{l}:\Phi _{e}^{\frak{s}}\rightarrow \Phi _{e}^{\widetilde{\mathfrak{s}}},
\end{equation*}
satisfying 
\begin{equation*}
F_{l}(\widetilde{D}^{m_l({\boldsymbol{\lambda}})})\simeq {D}^{%
\boldsymbol{\lambda}},
\end{equation*}
for all $\lambda \in \Phi _{e}^{\frak{s}}$.\ This map can be viewed as a
generalization of the Mullineux involution. In \cite{Fayers}, Fayers
describes this generalized Mullineux involution by using skew-isomorphisms
of $A_{e-1}^{(1)}$-crystals.\ This description is similar to a former one
obtained in \cite{LTV} for the analogue of the Mullineux involution, namely
the Zelevinsky involution, in affine Hecke algebras of type $A$.\ Both rest
on the determination of paths in crystal graphs.\ The aim of this note is to
give an alternative efficient description of the map $m_{l}$ in the spirit
of the original procedure of Mullineux, that is without using paths in
crystals.\ We obtain an algorithm to compute the Mullineux involution for
all positive integer $l$. This is achieved by combining the known
description of the original Mullineux involution (i.e. for $l=1$) with
results of \cite{JL} on isomorphisms of $A_{e-1}^{(1)}$-crystals.

The following sections are structured as follows. Section 2 is devoted to a
brief review on $A_{e-1}^{(1)}$-crystals and their labellings by Uglov and
Kleshchev $l$-partitions.\ In Section 3, we mainly recall and reformulate in
an appropriate language results of \cite{J} and \cite{JL}. Section 4
contains the main result of this paper, namely the description of the map $%
m_{l}$ on Kleshchev $l$-partitions.\newline

\section{$A_{e-1}^{(1)}$-crystals and the Mullineux involution}

Let ${\mathcal{U}_{v}(\widehat{\mathfrak{sl}_{e}})}$ be the quantum group of
affine type $A_{e-1}^{(1)}$. This is an associative $\mathbb{Q}(v)$-algebra
with generators $e_i,f_i,t_i,t_i^{-1}$ (for $i=0,...,e-1$) and $\partial$
and relations given in \cite[\S 2.1]{U}. We denote by ${\mathcal{U}_v
^{\prime}(\widehat{\mathfrak{sl}_e}) }$ the subalgebra generated by $%
e_i,f_i,t_i,t_i^{-1}$ (for $i=0,...,e-1$). We begin this section, by
reviewing some background on the crystal graph theory of the irreducible
highest weight ${\mathcal{U}_{v}^{\prime}(\widehat{\mathfrak{sl}_{e}})}$%
-modules. In particular, we recall the notion of Uglov and Kleshchev $l$%
-partitions and describe Fayers work on the Mullineux involution. We refer
to \cite{kashi} and to \cite{arikilivre} for the general theory of crystals. 
\cite[\S 7]{gecklivre} gives a nice survey on some of the problems we will
consider.

\subsection{Realizations of abstract $A_{e-1}^{(1)}$-crystals$\label%
{subsec_realization}$}

By slightly abuse of notation, we identify the elements of $\mathbb{Z}/e%
\mathbb{Z}$ with their corresponding labels in $\{0,...,e-1\}$ when there is
no risk of confusion. Let $v$ be an indeterminate and $e>1$ be an integer.
Write $\{\Lambda _{0},...,\Lambda _{e-1},\delta\}$ for the set of dominant
weights of ${\mathcal{U}_{v}(\widehat{\mathfrak{sl}_{e}})}$. Let $l\geq 1\in 
\mathbb{N}$ and consider ${\mathbf{s}}=(s_{0},...,s_{l-1})\in \mathbb{Z}^{l}$%
. We set $\mathfrak{s}=(s_{0}(\text{mod }e),...,s_{l-1}(\text{mod }e))\in (%
\mathbb{Z}/e\mathbb{Z})^{l}$ and $\Lambda _{\mathfrak{s}}:=\Lambda _{s_{0}(%
\text{mod }e)}+...+\Lambda _{s_{l-1}(\text{mod }e)}$.

As a $\mathbb{C(}v)$-vector space, the Fock space $\frak{F}_{e}$ of level $l$
admit the set of all $l$-partitions as a natural basis. Namely the
underlying vector space is 
\begin{equation*}
\frak{F}_{e}=\bigoplus_{n\geq 0}\bigoplus_{{\boldsymbol{\lambda}}\in \Pi
_{l,n}}\mathbb{C}(v){\boldsymbol{\lambda}},
\end{equation*}
where $\Pi _{l,n}$ is the set of $l$-partitions with rank $n$. One can put
different structures of ${\mathcal{U}_{v}(\widehat{\mathfrak{sl}_{e}})}$%
-modules on $\frak{F}_{e}$. To describe them, we need some combinatorics.

Let ${\boldsymbol{\lambda}}$ be an $l$-partition (identified with its Young
diagram). Then, the nodes of ${\boldsymbol{\lambda}}$ are the triplets $%
\gamma =(a,b,c)$ where $c\in \{0,...,l-1\}$ and $a,b$ are respectively the
row and column indices of the node $\gamma $ in $\lambda ^{(c)}.$ The
content of $\gamma $ is the integer $c\left( \gamma \right) =b-a+s_{c}$ and
the residue $\mathrm{res(}\gamma )$ of $\gamma $ is the element of $\mathbb{Z%
}/e\mathbb{Z}$ such that 
\begin{equation}
\mathrm{res}(\gamma )\equiv c(\gamma )(\text{mod }e).  \label{res}
\end{equation}
We will say that $\gamma $ is an $i$-node of ${\boldsymbol{\lambda}}$ when $%
\mathrm{res}(\gamma )\equiv i(\text{mod }e).$ Finally, We say that $\gamma $
is removable when $\gamma =(a,b,c)\in {\boldsymbol{\lambda}}$ and ${%
\boldsymbol{\lambda}}\backslash \{\gamma \}$ is an $l$-partition. Similarly $%
\gamma $ is addable when $\gamma =(a,b,c)\notin {\boldsymbol{\lambda}}$ and $%
{\boldsymbol{\lambda}}\cup \{\gamma \}$ is an $l$-partition.

We now define distinct total orders on the set of addable and removable $i$%
-nodes of the multipartitions. Consider $\gamma _{1}=(a_{1},b_{1},c_{1})$
and $\gamma _{2}=(a_{2},b_{2},c_{2})$ two $i$ -nodes in ${%
\boldsymbol{\lambda}}$. We define an order $\prec _{{\mathbf{s}}}$ by setting

\begin{equation*}
\gamma _{1}\prec _{{\mathbf{s}}}\gamma _{2}\Longleftrightarrow \left\{ 
\begin{array}{l}
c(\gamma _{1})<c(\gamma _{2})\text{ or} \\ 
c(\gamma _{1})=c(\gamma _{2})\text{ and }c_{1}>c_{2}
\end{array}
\right.
\end{equation*}

Let ${\boldsymbol{\lambda}}$ and ${\boldsymbol{\mu}}$ be two $\ell $%
-partitions of rank $n$ and $n+1$ and assume there exists an $i$-node $%
\gamma $ such that $[{\boldsymbol{\mu}}]=[{\boldsymbol{\lambda}}]\cup {%
\{\gamma \}}$. We define the following numbers: 
\begin{align*}
{N}_{i}^{\succ }{({\boldsymbol{\lambda}},{\boldsymbol{\mu}})}=& \sharp \{%
\text{addable }\ i\text{-nodes }\gamma \text{ of }{\boldsymbol{\lambda}}\ 
\text{ such that }\gamma ^{\prime }\succ _{{\mathbf{s}}}\gamma \} \\
& -\sharp \{\text{removable }\ i\text{-nodes }\gamma ^{\prime }\text{ of }{%
\boldsymbol{\mu}}\ \text{ such that }\gamma ^{\prime }\succ _{{\mathbf{s}}%
}\gamma \}, \\
{N}_{i}^{\prec }{({\boldsymbol{\lambda}},{\boldsymbol{\mu}})}=& \sharp \{%
\text{addable }i\text{-nodes }\gamma ^{\prime }\text{ of }{%
\boldsymbol{\lambda}}\ \text{ such that }\gamma ^{\prime }\prec _{{\mathbf{s}%
}}\gamma \} \\
& -\sharp \{\text{removable }i\text{-nodes }\gamma ^{\prime }\text{ of }{%
\boldsymbol{\mu}}\ \text{ such that }\gamma ^{\prime }\prec _{{\mathbf{s}}%
}\gamma \}, \\
{N}_{i}{({\boldsymbol{\lambda}})}=& \sharp \{\text{addable }i\text{-nodes of 
}{\boldsymbol{\lambda}}\} \\
& -\sharp \{\text{removable }i\text{-nodes of }{\boldsymbol{\lambda}}\}\text{
for }0\leq i\leq e-1 \\
\text{ and }{M}_{0}{({\boldsymbol{\lambda}})}=& \sharp \{0\text{-nodes of }{%
\boldsymbol{\lambda}}\}.
\end{align*}

The following theorem is \cite[Thm. 2.1]{U}.

\begin{thm}[Jimbo et al. \protect\cite{jim}, Foda et al. \protect\cite{FLOTW}%
, Uglov \protect\cite{U}]
\label{jmm} Let $\mathbf{s}\in \mathbb{Z}^{l}$. Then we can put an action of 
${\mathcal{U}_{v}(\widehat{\mathfrak{sl}_{e}})}$ on $\mathcal{F}_{e}$ as
follows. 
\begin{equation*}
e_{i}{\boldsymbol{\lambda}}=\sum_{\text{res}([{\boldsymbol{\lambda}}]/[{%
\boldsymbol{\mu}}])=i}{q^{-{N}_{i}^{\succ }{({\boldsymbol{\mu}},{%
\boldsymbol{\lambda}})}}{\boldsymbol{\mu}},\qquad {f_{i}{\boldsymbol{\lambda}%
}=\sum_{\text{res}([{\boldsymbol{\mu}}]/[{\boldsymbol{\lambda}}])=i}{q^{{N}%
_{i}^{\prec }{({\boldsymbol{\lambda}},{\boldsymbol{\mu}})}}}}}{%
\boldsymbol{\mu}},
\end{equation*}
\begin{equation*}
t_{i}{\boldsymbol{\lambda}}=q^{{N}_{i}{({\boldsymbol{\lambda}})}}{%
\boldsymbol{\lambda}},\qquad \partial {\boldsymbol{\lambda}}=-(\Delta +M_{0}(%
{\boldsymbol{\lambda}})){\boldsymbol{\lambda}},\qquad (0\leq i\leq e-1)
\end{equation*}
where $\Delta $ is a rational number defined in \cite[Thm 2.1]{U} which does
not depend on ${\boldsymbol{\lambda}}$ (but depends on $\mathbf{s})$. The
associated ${\mathcal{U}_{v}(\widehat{\mathfrak{sl}_{e}})}$-module which is
denoted by $\mathcal{F}_{e}^{\mathbf{s}}$ is an integrable module.
\end{thm}

Now, one can define another structure of ${\mathcal{U}_{v}(\widehat{%
\mathfrak{sl}_{e}})}$-module on the Fock space $\mathcal{F}_{e}$. This is
achieved by considering the following order on the set of $i$-nodes of an $l$%
-partition : 
\begin{equation*}
\gamma _{1}\prec _{\frak{s}}\gamma _{2}\Longleftrightarrow \left\{ 
\begin{array}{l}
c_{1}>c_{2}\text{ or} \\ 
c_{1}=c_{2}\text{ and }c(\gamma _{1})<c(\gamma _{2})
\end{array}
\right. .
\end{equation*}
Using the same formulas as in theorem \ref{jmm} where $\prec _{{\mathbf{s}}}$
is replaced by $\prec _{\frak{s}}$ and $\Delta=0$, the space $\mathcal{F}_{e}
$ is endowed with a different structure of an integrable ${\mathcal{U}_{v}(%
\widehat{\mathfrak{sl}_{e}})}$-module that we will denote by $\mathcal{F}%
_{e}^{\mathfrak{ s}}$. We refer to \cite[\S 10.2]{arikilivre} for the proof
of this assertion.

Observe that the order $\prec _{{\mathbf{s}}}$ and thus $\mathcal{F}_{e}^{%
\mathbf{s}}$ really depends on ${\mathbf{s}}$ whereas $\prec _{\frak{s}}$
and thus $\mathcal{F}_{e}^{\mathfrak{ s}}$ only depends on the class $%
\mathfrak{s}\in (\mathbb{Z}/e\mathbb{Z})^{l}$.

\begin{rem}
\ 

\begin{enumerate}
\item  For $l=1$, the Fock spaces $\frak{F}_{e}^{\mathbf{s}}$ and $\frak{F}%
_{e}^{\frak{s}}$ coincide.

\item  With $\mathfrak{s}=(s_{0}(\text{mod }e),...,s_{l-1}(\text{mod }e))$),
we have 
\begin{equation}
\frak{F}_{e}^{\frak{s}}=\frak{F}_{e}^{s_{0}}\otimes \cdots \otimes \frak{F}%
_{e}^{s_{l-1}}.  \label{Fock_tens}
\end{equation}
The Fock space $\frak{F}_{e}^{\frak{s}}$ is a tensor product of Fock spaces
of level $1$.
\end{enumerate}
\end{rem}

The ${\mathcal{U}_v (\widehat{\mathfrak{sl}_e}) }$-submodule of $\frak{F}%
_{e}^{{\mathbf{s}}}$ generated by the empty $l$-partition is denoted by $%
V_{e}^{{\mathbf{s}}}(\Lambda _{\mathfrak{s}})$. By \cite[Rem. 2.2]{U}, this
is an irreducible highest weight ${\mathcal{U}_v (\widehat{\mathfrak{sl}_e}) 
}$-module with weight $\Lambda_{{\mathbf{s}}}-\Delta \delta$. Similarly, the 
${\mathcal{U}_v (\widehat{\mathfrak{sl}_e}) }$-submodule of $\frak{F}_{e}^{{%
\mathfrak{s}}}$ generated by the empty $l$-partition is denoted by $V_{e}^{{%
\mathfrak{s}}}(\Lambda _{\mathfrak{s}})$. By \cite[thm 10.10]{arikilivre},
this is an irreducible highest weight ${\mathcal{U}_v (\widehat{\mathfrak{sl}%
_e}) }$-module with weight $\Lambda_{\frak{s}}$.

\noindent We rather need in the sequel the restrictions $V_{e}^{{\mathbf{s}}%
}(\Lambda _{\mathfrak{s}})^{\prime }$ and $V_{e}^{{\mathfrak{s}}}(\Lambda _{%
\mathfrak{s}})^{\prime }$ of $V_{e}^{{\mathbf{s}}}(\Lambda _{\mathfrak{s}})$
and $V_{e}^{{\mathfrak{s}}}(\Lambda _{\mathfrak{s}})$ to the subalgebra ${%
\mathcal{U}_{v}^{\prime }(\widehat{\mathfrak{sl}_{e}})}$. They are both
irreducible as ${\mathcal{U}_{v}^{\prime }(\widehat{\mathfrak{sl}_{e}})}$%
-modules of highest weight $\Lambda _{\mathfrak{s}}$. In particular, all the
modules $V_{e}^{{\mathbf{s}}}(\Lambda _{\mathfrak{s}})^{\prime },$ ${\mathbf{%
s}}\in \frak{s}$ are isomorphic to $V_{e}^{{\mathfrak{s}}}(\Lambda _{%
\mathfrak{s}})^{\prime }.$ However, as explained above, the corresponding
actions of the Chevalley operators on these modules do not coincide in
general.

The modules $\frak{F}_{e}^{{\mathbf{s}}}$ and $\frak{F}_{e}^{\frak{s}}$ are
integrable. Hence, by the general theory of Kashiwara crystal bases, $\frak{F%
}_{e}^{{\mathbf{s}}}$ and $\frak{F}_{e}^{\frak{s}}$ have crystal graphs that
we denote $B_{e}^{{\mathbf{s}}}$ and $B_{e}^{\frak{s}},$ respectively. Since
the module structures $\frak{F}_{e}^{{\mathbf{s}}}$ and $\frak{F}_{e}^{\frak{%
s}}$ depend on the total orders $\prec _{{\mathbf{s}}}$ and $\prec _{\frak{s}%
}$, the graph structures on $B_{e}^{{\mathbf{s}}}$ and $B_{e}^{\frak{s}}$ do
not coincide in general although they both admit the set of $l$-partitions
as set of vertices.

\noindent To describe the graph structure of $B_{e}^{{\mathbf{s}}}$ (for a
fixed $\mathbf{s}\in {\frak{s}}$) we proceed as follows. Starting from any $%
l $-partition ${\boldsymbol{\lambda}}$, we can consider its set of addable
and removable $i$-nodes. Let $w_{i}$ be the word obtained first by writing
the addable and removable $i$-nodes of ${\boldsymbol{\lambda}}$ in {%
increasing} order with respect to $\prec _{{\mathbf{s}}}$ next by encoding
each addable $i$-node by the letter $A$ and each removable $i$-node by the
letter $R$.\ Write $\widetilde{w}_{i}=A^{p}R^{q}$ for the word derived from $%
w_{i}$ by deleting as many of the factors $RA$ as possible. If $p>0,$ let $%
\gamma $ be the rightmost addable $i$-node in $\widetilde{w}_{i}$.\ When $%
\widetilde{w}_{i}\neq \emptyset $, the node $\gamma $ is called the good $i$%
-node. Now, the crystal $B_{e}^{{\mathbf{s}}}$ is the graph with

\begin{itemize}
\item  vertices: the $l$-partitions,

\item  edges: $\displaystyle{{\boldsymbol{\lambda}}}\overset{i}{\rightarrow }%
{{\boldsymbol{\mu}}}$ if and only if ${\boldsymbol{\mu}}$ is obtained by
adding to ${\boldsymbol{\lambda}}$ a good $i$-node.
\end{itemize}

\noindent The graph structure on and $B_{e}^{\frak{s}}$ is obtained
similarly by using the order $\prec _{\frak{s}}$ instead of $\prec _{{%
\mathbf{s}}}$.

\subsection{Uglov and Kleshchev multipartitions}

By definition $V_{e}^{{\mathbf{s}}}(\Lambda _{\frak{s}})$ and $V_{e}^{\frak{s%
}}(\Lambda _{\frak{s}})$ are the irreducible components with highest weight
vector $\boldsymbol{
\emptyset }$ in $\frak{F}_{e}^{{\mathbf{s}}}$ and $\frak{F}_{e}^{\frak{s}}$,
their crystal graphs $B_{e}^{{\mathbf{s}}}(\Lambda _{{\frak{s}}})$ and $%
B_{e}^{\frak{s}}(\Lambda _{{\frak{s}}})$ (which are also the crystal graphs
of $V_{e}^{{\mathbf{s}}}(\Lambda _{\frak{s}})^{\prime }$ and $V_{e}^{\frak{s}%
}(\Lambda _{\frak{s}})^{\prime }$ ) can be realized respectively as the
connected components of highest weight vertex $\boldsymbol{\emptyset }$ in $%
B_{e}^{\boldsymbol{s}}$ and $B_{e}^{\frak{s}}$.

\begin{exmp}
The graph below is the subgraph of $B_{4}^{(0,0,1)}(2\Lambda _{0}+\Lambda
_{1})$ containing the $3$-partitions of rank less or equal to $4$.
\end{exmp}

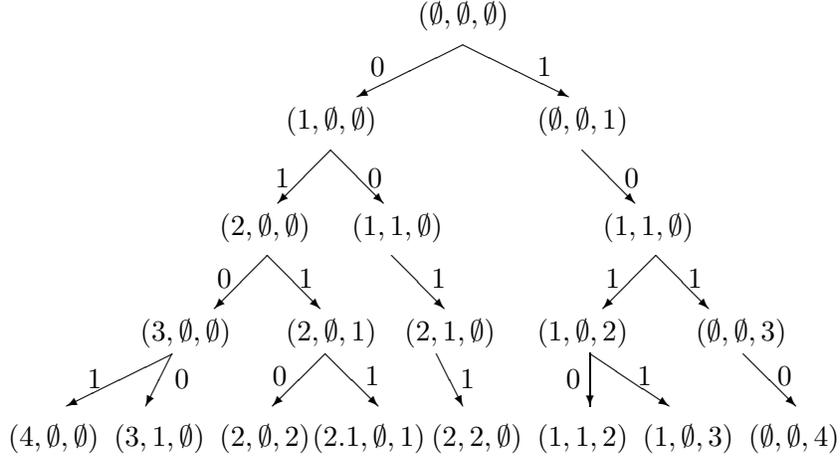
\begin{center}
\begin{picture}(250,150)
\put(  125, 140){$(\emptyset,\emptyset,\emptyset)$}

 \put( 75, 100){$(1,\emptyset,\emptyset)$}
  \put( 170, 100){$(\emptyset,\emptyset,1)$}

\put( 50, 60){$(2,\emptyset,\emptyset)$}
\put( 100, 60){$(1,1,\emptyset)$} 
\put( 195, 60){$(1,1,\emptyset)$}

\put( 20, 20){$(3,\emptyset,\emptyset)$}
\put( 75, 20){$(2,\emptyset,1)$} 
\put( 120, 20){$(2,1,\emptyset)$} 
\put( 170, 20){$(1,\emptyset,2)$} 
\put( 230, 20){$(\emptyset,\emptyset,3)$}

\put( -30, -20){$(4,\emptyset,\emptyset)$}
\put( 10, -20){$(3,1,\emptyset)$} 
\put( 50, -20){$(2,\emptyset,2)$} 
\put( 85,- 20){$(2.1,\emptyset,1)$} 
\put( 130, -20){$(2,2,\emptyset)$} 
\put( 170,-20){$(1,1,2)$}
\put( 210, -20){$(1,\emptyset,3)$} 
\put( 250,-20){$(\emptyset,\emptyset,4)$}

 \put( 142, 132){\vector(-2,-1){40}} 
  \put(142, 132){\vector(2,-1){40}}

 \put( 92, 92){\vector(-1,-1){20}} \put( 92, 92){\vector(1,-1){20}} 
   \put(187, 92){\vector(1,-1){20}}

 \put( 68, 52){\vector(-1,-1){20}} \put( 68, 52){\vector(1,-1){20}} 
   \put(115, 52){\vector(1,-1){20}} 
 \put( 215, 52){\vector(-1,-1){20}} \put( 215, 52){\vector(1,-1){20}}

 \put( 32, 15){\vector(-2,-1){40}} \put( 32, 15){\vector(-1,-2){10}}

  \put( 132, 15){\vector(1,-2){10}} 
   \put(90, 15){\vector(-1,-1){20}}   \put(90, 15){\vector(1,-1){20}} 
 \put( 190, 15){\vector(0,-1){20}} \put( 190, 15){\vector(3,-2){30}} 
\put( 248, 15){\vector(1,-1){20}}

\put(  107, 120){$0$} 
\put(  170, 120){$1$}

\put(  203, 78){$0$} 
\put(  106, 78){$0$} 
 \put(  71, 78){$1$}

\put(  196, 40){$1$} 
 \put(  227, 40){$1$}  
\put(  80, 40){$1$} 
\put(  49, 40){$0$} 
\put(  130, 40){$1$}

\put(  261, 3){$0$} 
\put(  181, 2){$0$} 
\put(  141, 2){$1$} 
\put(  208, 3){$1$} 
 \put(  0, 2){$1$} 
\put(  33, 2){$0$} 
\put(  70, 3){$0$} 
 \put(  105, 3){$1$} 
\end{picture} 
\end{center}

\vspace{1cm}

We denote by $\Phi _{e}^{{\mathbf{s}}}$ the set of vertices in $B_{e}^{{%
\mathbf{s}}}(\Lambda _{{\frak{s}}})$. The elements of $\Phi _{e}^{{\mathbf{s}%
}}$ are called the Uglov $l$-partitions. We also write $\Phi _{e}^{{\mathbf{s%
}}}(n)$ for the subset of $\Phi _{e}^{{\mathbf{s}}}$ containing the Uglov $l$%
-partitions with rank $n$. In general the set $\Phi _{e}^{{\mathbf{s}}}$ is
very difficult to describe without computing the underlying crystal.
Nevertheless, in the case where $0\leq s_{0}\leq \cdot \cdot \cdot \leq
s_{l-1}\leq e-1,$ Foda et al. have given a simple non recursive description
of the Uglov $l$-partitions (see \cite[Thm 2.10]{FLOTW}).

We denote by $\Phi _{e}^{\frak{s}}$ the set of vertices in $B_{e}^{\frak{s}%
}(\Lambda _{{\frak{s}}}).$ The elements of $\Phi _{e}^{\frak{s}}$ are called
the Kleshchev $l$-partitions. We define $\Phi _{e}^{\frak{s}}(n)$ as we have
done in the Uglov case. Note that there exists only a recursive description
of the Kleshchev $l$-partitions $\Phi _{e}^{\frak{s}},$ except when the
number of fundamental weights appearing in the decomposition of $\Lambda _{{%
\frak{s}}}$ is less or equal to $2$ , see \cite{AJ} and \cite{Ari3}. In the
other cases, they can only be obtained by computing the underlying crystal
or using the procedure described in \cite[\S 5.4]{JL}.

\begin{exmp}
The graph below is the subgraph of $B_{4}^{\frak{s}}(2\Lambda _{0}+\Lambda
_{1})$ with $\frak{s}=(0$ $($mod $4),0$ $($mod $4),1$ $($mod $4$))
containing the $3$-partitions of rank less than $4$.
\end{exmp}

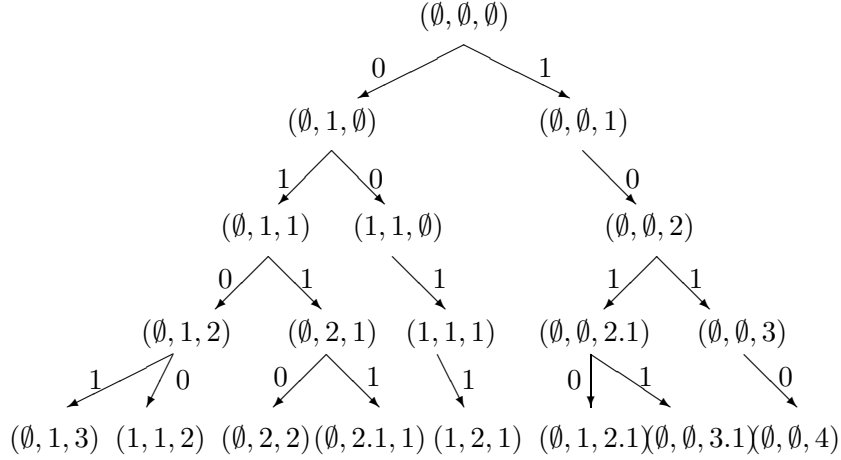
\begin{center}
\begin{picture}(250,150)
\put(  125, 140){$(\emptyset,\emptyset,\emptyset)$}

 \put( 75, 100){$(\emptyset,1,\emptyset)$}
  \put( 170, 100){$(\emptyset,\emptyset,1)$}

\put( 50, 60){$(\emptyset,1,1)$}
\put( 100, 60){$(1,1,\emptyset)$} 
\put( 195, 60){$(\emptyset,\emptyset,2)$}

\put( 20, 20){$(\emptyset,1,2)$}
\put( 75, 20){$(\emptyset,2,1)$} 
\put( 120, 20){$(1,1,1)$} 
\put( 170, 20){$(\emptyset,\emptyset,2.1)$} 
\put( 230, 20){$(\emptyset,\emptyset,3)$}

\put( -30, -20){$(\emptyset,1,3)$}
\put( 10, -20){$(1,1,2)$} 
\put( 50, -20){$(\emptyset,2,2)$} 
\put( 85,- 20){$(\emptyset,2.1,1)$} 
\put( 130, -20){$(1,2,1)$} 
\put( 170,-20){$(\emptyset,1,2.1)$}
\put( 210, -20){$(\emptyset,\emptyset,3.1)$} 
\put( 250,-20){$(\emptyset,\emptyset,4)$}

 \put( 142, 132){\vector(-2,-1){40}} 
  \put(142, 132){\vector(2,-1){40}}

 \put( 32, 15){\vector(-2,-1){40}} \put( 32, 15){\vector(-1,-2){10}}

 \put( 92, 92){\vector(-1,-1){20}} \put( 92, 92){\vector(1,-1){20}} 
   \put(187, 92){\vector(1,-1){20}}

 \put( 68, 52){\vector(-1,-1){20}} \put( 68, 52){\vector(1,-1){20}} 
   \put(115, 52){\vector(1,-1){20}} 
 \put( 215, 52){\vector(-1,-1){20}} \put( 215, 52){\vector(1,-1){20}}

  \put( 132, 15){\vector(1,-2){10}} 
   \put(90, 15){\vector(-1,-1){20}}   \put(90, 15){\vector(1,-1){20}} 
 \put( 190, 15){\vector(0,-1){20}} \put( 190, 15){\vector(3,-2){30}} 
\put( 248, 15){\vector(1,-1){20}}

\put(  107, 120){$0$} 
\put(  170, 120){$1$}

\put(  203, 78){$0$} 
\put(  106, 78){$0$} 
 \put(  71, 78){$1$}

\put(  196, 40){$1$} 
 \put(  227, 40){$1$}  
\put(  80, 40){$1$} 
\put(  49, 40){$0$} 
\put(  130, 40){$1$}

\put(  261, 3){$0$} 
\put(  181, 2){$0$} 
\put(  141, 2){$1$} 
\put(  208, 3){$1$} 
 \put(  0, 2){$1$} 
\put(  33, 2){$0$} 
\put(  70, 3){$0$} 
 \put(  105, 3){$1$} 

\end{picture} 
\end{center}

\vspace{1cm}

Let ${\mathbf{s}}^{\prime }:=(s_{0}^{\prime },...,s_{l-1}^{\prime })\in 
\mathbb{Z}^{l}$ be such that ${\mathbf{s}}^{\prime }\in \mathfrak{s}$. Then
the two modules $V_{e}^{{\mathbf{s}}}(\Lambda _{\frak{s}})^{\prime}$ and $%
V_{e}^{{\mathbf{s}}^{\prime }}(\Lambda _{\frak{s}})$ are both isomorphic to $%
V_{e}^{\frak{s}}(\Lambda _{\frak{s}})^{\prime}$. Hence the corresponding
three crystal graphs are isomorphic. These crystal isomorphisms define the
following bijections 
\begin{equation*}
\Psi _{e,n}^{{\mathbf{s}}\rightarrow {\mathbf{s}}^{\prime }}:\Phi _{e}^{{%
\mathbf{s}}}(n)\rightarrow \Phi _{e}^{{\mathbf{s}}^{\prime }}(n)
\end{equation*}
\begin{equation*}
\Psi _{e,n}^{{\mathbf{s}}\rightarrow \frak{s}}:\Phi _{e}^{{\mathbf{s}}%
}(n)\rightarrow \Phi _{e}^{\frak{s}}(n).
\end{equation*}
As it can be easily checked on the previous examples, the sets $\Phi _{e}^{%
\mathbf{s}}(n)$ and $\Phi _{e}^{\frak{s}}(n)$ do not coincide in general.\
However, if we fix $n$, it is possible to realize the set $\Phi _{e}^{\frak{s%
}}(n)$ of Kleshchev $l$-partitions with rank $n$ as a special set $\Phi
_{e}^{{{\mathbf{s}}}_{\infty }}(n)$, ${{\mathbf{s}}}_{\infty }\in \frak{s}$
of Uglov $l$-partitions. We refer the reader to Lemma 5.4.1 in \cite{JL} for
the proof of the following proposition.

\begin{prop}
\label{ident} Let $n\in \mathbb{N}$ and consider ${\mathbf{s}}_{\infty
}:=(s_{0},...,s_{l-1})\in \mathbb{Z}^{l}$ such that 
\begin{equation}
s_{i+1}-s_{i}>n-1  \label{asymp}
\end{equation}
for all $i=0,...,l-2$. Then $\Psi _{e,n}^{{\mathbf{s}}_{\infty }\rightarrow 
\frak{s}}$ is the identity. Hence, we have : 
\begin{equation*}
\Phi _{e}^{\frak{s}}(n)=\Phi _{e}^{{{\mathbf{s}}}_{\infty }}(n).
\end{equation*}
In particular, $\Phi _{e}^{{{\mathbf{s}}}_{\infty }}(n)$ does not depend on
the multicharge ${\mathbf{s}}_{\infty }$ provided (\ref{asymp}) is satisfied.
\end{prop}

This shows that the description of the bijections $\Psi _{e,n}^{{\mathbf{s}}%
\rightarrow {\mathbf{s}}^{\prime }}$ for all ${\mathbf{s}},{\mathbf{s}}%
^{\prime }\in \mathfrak{s}$ will lead that of $\Psi _{e,n}^{{\mathbf{s}}%
\rightarrow \frak{s}}$ as a special case.\ A multicharge verifying (\ref
{asymp}) is called asymptotic (this is called $n$-dominant in \cite{Y}).
This justifies the notation $\Phi _{e}^{{{\mathbf{s}}}_{\infty }}(n)$ we
have adopted.

\section{Isomorphisms of $A_{e-1}^{(1)}$-crystals\label{iso}}

We now review some results obtained in \cite{J} and \cite{JL} making
explicit the bijections $\Psi _{e,n}^{{\mathbf{s}}\rightarrow {\mathbf{s}}%
^{\prime }}$. In \cite{JL}, we have chosen to explain our results in the
language of column shaped tableaux to make apparent the link with the
combinatorics of $\frak{sl}_{n}$. Here, we rather adopt the equivalent
language of symbols also used in \cite{J} which permits to give a very
simple description of the elementary steps involved in our algorithms.

\subsection{Elementary transformations in $\widehat{S}_{l}$\label{sub_secW}}

We write as usual $\widehat{S}_{l}$ for the extended affine symmetric group
in type $A_{l-1}.$ This group is generated by the elements $\sigma
_{1},...,\sigma _{l-1}$ and $y_{0},....,y_{l-1}$ together with the relations 
\begin{equation*}
\sigma _{c}\sigma _{c+1}\sigma _{c}=\sigma _{c+1}\sigma _{c}\sigma
_{c+1},\quad \sigma _{c}\sigma _{d}=\sigma _{d}\sigma _{c}\text{ for }\left|
c-d\right| >1,\quad \sigma _{c}^{2}=1,
\end{equation*}
with all indices in $\{1,...,l-1\}$ and 
\begin{equation*}
y_{c}y_{d}=y_{d}y_{c},\quad \sigma _{c}y_{d}=y_{d}\sigma _{c}\text{ for }%
d\neq c,c+1,\quad \sigma _{c}y_{c}\sigma _{c}=y_{c+1}.
\end{equation*}
For any $c\in \{0,...,l-1\}$, we set $z_{c}=y_{1}\cdot \cdot \cdot y_{c}.$
Write also $\xi =\sigma _{l-1}\cdot \cdot \cdot \sigma _{1}$ and $\tau
=y_{l}\xi .$ Since $y_{c}=z_{c-1}^{-1}z_{c}$, $\widehat{S}_{l}$ is generated
by the transpositions $\sigma _{c}$ with $c\in \{1,...,l-1\}$ and the
elements $z_{c}$ with $c\in \{1,...,l\}.$ Observe that for any $c\in
\{1,...,l-1\},$ we have 
\begin{equation}
z_{c}=\xi ^{l-c}\tau ^{c}.  \label{dec-zi}
\end{equation}
This implies that $\widehat{S}_{l}$ is generated by the transpositions $%
\sigma _{c}$ with $c\in \{1,...,l-1\}$ and $\tau .$ Consider $e$ a fixed
positive integer. We obtain a faithful action of $\widehat{S}_{l}$ on $%
\mathbb{Z}^{l}$ by setting for any ${\mathbf{s}}=(s_{0},...,s_{l-1})\in 
\mathbb{Z}^{l}$%
\begin{equation*}
\sigma _{c}({\mathbf{s}})=(s_{0},...,s_{c},s_{c-1},...,s_{l-1})\text{ and }%
y_{c}({\mathbf{s}})=(s_{0},...,s_{c-1},s_{c}+e,...,s_{l-1}).
\end{equation*}
Then $\tau ({\mathbf{s}})=(s_{1},s_{2},...,s_{l-1},s_{0}+e)$.

\begin{rem}
>From the preceding considerations, $\{s_{1},\ldots ,s_{l-1},\tau \}$ is a
minimal set of generators for $\widehat{S}_{l}$.\ Hence, the bijections $%
\Psi _{e,n}^{{\mathbf{s}}\rightarrow {\mathbf{s}}^{\prime }}:\Phi _{e}^{{%
\mathbf{s}}}(n)\rightarrow \Phi _{e}^{{\mathbf{s}}^{\prime }}(n)$ with ${%
\mathbf{s}}\in \mathfrak{s}$ and $\Psi _{e,n}^{{\mathbf{s}}\rightarrow \frak{%
s}}:\Phi _{e}^{{\mathbf{s}}}(n)\rightarrow \Phi _{e}^{\frak{s}}(n)$ can be
obtained by composing bijections corresponding to the cases ${\mathbf{s}}%
^{\prime }=\tau ({\mathbf{s}})$ and ${\mathbf{s}}^{\prime }=\sigma _{c}({%
\mathbf{s}})$ with $c=1,...,l-1$.
\end{rem}

\subsection{Elementary crystals isomorphisms\label{subsec-elemcrys}}

The following proposition is stated in \cite[Prop. 5.2.1]{JL}. We give the
proof for the convenience of the reader.

\begin{prop}
Let ${\mathbf{s}}\in \mathbb{Z}^{l}$ and $e\in \mathbb{N}$. Let ${%
\boldsymbol{\lambda}}=(\lambda ^{(0)},...,\lambda ^{(l-1)})\in \Phi _{e}^{{%
\mathbf{s}}}(n)$. Then 
\begin{equation*}
\Psi _{e,n}^{{\mathbf{s}}\rightarrow \tau ({\mathbf{s}})}=(\lambda
^{(1)},...,\lambda ^{(l-1)},\lambda ^{(0)}).
\end{equation*}
\end{prop}

\begin{proof}
Set $\mathbf{s}=(s_{0},...,s_{\ell -1}).\;$Then $\tau (\mathbf{s}%
)=(s_{1},....,s_{\ell -1},s_{0}+e)$. Consider ${\ulambda}=(\lambda
^{(0)},...,\lambda ^{(\ell -1)})$ a multipartition and set ${\ulambda}%
^{\diamondsuit }=(\lambda ^{(1)},...,\lambda ^{(\ell -1)},\lambda ^{(0)})$.
Consider $i\in \{0,1,...,e-1\}$ and $\gamma _{1}=(a_{1},b_{1},c_{1}),$ $%
\gamma _{2}=(a_{2},b_{2},c_{2})$ two $i$-nodes of ${\ulambda}$. Then $\gamma
_{1}^{\diamondsuit }=(a_{1},b_{1},c_{1}-1(\text{mod }e))$ and $\gamma
_{2}^{\diamondsuit }=(a_{2},b_{2},c_{2}-1(\text{mod }e))$ are two $i$-nodes
of ${\ulambda}^{\diamondsuit }$. We then easily check that $\gamma _{2}\prec
_{\mathbf{s}}\gamma _{1}$ if and only if $\gamma _{2}^{\diamondsuit }\prec
_{\tau (\mathbf{s})}\gamma _{1}^{\prime }.$ This implies that $\Psi _{e,n}^{%
\mathbf{s}\rightarrow \tau (\mathbf{s})}({\ulambda})={\ulambda}%
^{\diamondsuit }$.
\end{proof}

The description of the bijections $\Psi _{e,n}^{{\mathbf{s}}\rightarrow
\sigma _{c}({\mathbf{s}})},$ $c=1,...,l-1$ are more complicated. It
essentially rests on the following basic procedure on pairs $(U,D)$ of
strictly increasing sequences of integers. let $U=[x_{1},\ldots ,x_{r}]$ and 
$D=[y_{1},\ldots ,y_{s}]$ two strictly increasing sequences of integers. We
compute from $(U,D)$ a new pair $(U^{\prime },D^{\prime })$ with $U^{\prime
}=[x_{1}^{\prime },\ldots ,x_{r}^{\prime }]$ and $D^{\prime }=[y_{1}^{\prime
},\ldots ,y_{s}^{\prime }]$ of such sequences by applying the following
algorithm :

\bigskip

\noindent \textbf{Algorithm.}

\begin{itemize}
\item  Assume $r\geq s$. We associate to $y_{1}$ the integer $x_{i_{1}}\in U$
such that 
\begin{equation}
x_{i_{1}}=\left\{ 
\begin{array}{l}
\mathrm{max}\{x\in U\mid y_{1}\geq x\}\text{ if }y_{1}\geq x_{1} \\ 
x_{r}\text{ otherwise}
\end{array}
\right. .  \label{algo1}
\end{equation}
We repeat the same procedure to the ordered pair $(U\setminus
\{x_{i_{1}}\},D\setminus \{y_{1}\}).$ By induction this yields a subset $%
\{x_{i_{1}},\ldots ,x_{i_{s}}\}\subset U.\;$Then we define $D^{\prime }$ as
the increasing reordering $\{x_{i_{1}},\ldots ,x_{i_{s}}\}$ and $U^{\prime }$
as the increasing reordering of $U\setminus \{x_{i_{1}},\ldots
,x_{i_{s}}\}\sqcup D.$

\item  Assume $r<s.\;$We associate to $x_{1}$ the integer $y_{i_{1}}\in D$
such that 
\begin{equation}
y_{i_{1}}=\left\{ 
\begin{array}{l}
\mathrm{min}\{y\in D\mid x_{1}\leq y\}\text{ if }x_{1}\leq y_{s} \\ 
y_{1}\text{ otherwise}
\end{array}
\right. .  \label{algo2}
\end{equation}
We repeat the same procedure to the ordered sequences $U\setminus \{x_{1}\}$
and $D\setminus \{y_{i_{1}}\}$ and obtain a subset $\{y_{i_{1}},\ldots
,y_{i_{r}}\}\subset D.\;$Then we define $U^{\prime }$ as the increasing
reordering $\{y_{i_{1}},\ldots ,y_{i_{r}}\}$ and $D^{\prime }$ as the
increasing reordering of $D\setminus \{y_{i_{1}},\ldots ,y_{i_{r}}\}\sqcup
U. $
\end{itemize}

\begin{exmp}
Take $U=[0,1,3]$ and $D=[0,1,2,4,7]$. The subset of $U$ determined by the
previous algorithm is $\{0,1,4\}$. This gives $U^{\prime }=[0,1,4]$ and $%
D^{\prime }=[0,1,2,3,7].$
\end{exmp}

Now consider $c\in \{1,\ldots ,e-1\}$ and ${\boldsymbol{\lambda}}=(\lambda
^{(0)},...,\lambda ^{(l-1)})\in \Phi _{e}^{{\mathbf{s}}}(n)$.\ In order to
compute $\Psi _{e,n}^{{\mathbf{s}}\rightarrow \sigma _{c}({\mathbf{s}})}({%
\boldsymbol{\lambda})}$, we need the symbol $S_{c}$ of the bipartition $%
(\lambda ^{(c-1)},\lambda ^{(c)})$. Let $d_{c-1}$ and $d_{c}$ be the number
of nonzero parts in $\lambda ^{(c-1)}$ and $\lambda ^{(c)}$ and put $m=%
\mathrm{max}(d_{c-1}-s_{c-1},d_{c}-s_{c})+1.$ Set 
\begin{equation*}
\left\{ 
\begin{array}{l}
\beta _{i}^{(c-1)}=\lambda _{i}^{(c-1)}-i+s_{c-1}+m\text{ for any }%
i=1,\ldots ,m+s_{c-1} \\ 
\beta _{i}^{(c)}=\lambda _{i}^{(c)}-i+s_{c}+m\text{ for any }i=1,\ldots
,m+s_{c}
\end{array}
\right.
\end{equation*}
where, by convention the partitions $\lambda ^{(c-1)}$ and $\lambda ^{(c)}$
are taken with an infinite number of zero parts. The symbol $S_{c}$ is the
two-row tableau 
\begin{equation*}
S_{c}=\left( 
\begin{array}{cccccc}
\beta _{m+s_{c}}^{(c)} & \beta _{m+s_{c}-1}^{(c)} & ... & ... & ... & \beta
_{1}^{(c)} \\ 
\beta _{m+s_{c-1}}^{(c-1)} & \beta _{m+s_{c-1}-1}^{(c-1)} & ... & ... & 
\beta _{1}^{(c-1)} & 
\end{array}
\right) =\left( 
\begin{array}{c}
U \\ 
D
\end{array}
\right)
\end{equation*}
Observe that $\beta _{i}^{(k)},k\in \{c-1,c\}$ is nothing but the content of
the node $\gamma =(i,\lambda _{i}^{(k)},k)\in {\boldsymbol{\lambda}}$
translated by $m$.\ This translation normalizes the symbol $S_{c}$ so that $%
\beta _{m+s_{c}}^{(c)}=\beta _{m+s_{c-1}}^{(c-1)}=0$.\ In fact, for a fixed
pair $(s_{c-1},s_{c})$ and $m$ as above, the map $\Sigma _{(s_{c-1},s_{c})}$
which associates to each bipartition its symbol is bijective.\ It is easy,
from a symbol $S_{c}$, to recover the unique bipartition $(\lambda
^{(c-1)},\lambda ^{(c)})$ such that $S_{c}$ is the symbol of $(\lambda
^{(c-1)},\lambda ^{(c)}).\;$We are now ready to describe the bijection $\Psi
_{e,n}^{{\mathbf{s}}\rightarrow \sigma _{c}({\mathbf{s}})}:\Phi _{e}^{{%
\mathbf{s}}}(n)\rightarrow \Phi _{e}^{\sigma _{c}({\mathbf{s}})}(n)$.

\noindent Let ${\mathbf{s}}\in \mathbb{Z}^{l}$ and $e\in \mathbb{N}$.
Consider ${\boldsymbol{\lambda}}=(\lambda ^{(0)},...,\lambda ^{(l-1)})\in
\Phi _{e}^{{\mathbf{s}}}(n)$ and $c\in \{1,...,l-1\}$. Write $S_{c}=\binom{U%
}{D}$ for the symbol corresponding to the bipartition $(\lambda
^{(c-1)},\lambda ^{(c)})$. Denote by $\widetilde{S}_{c}=\binom{\widetilde{U}%
}{\widetilde{D}}$ the symbol obtained by applying to $S_{c}$ the previous
algorithm. Then we compute the bipartition $(\widetilde{\lambda }^{(c-1)},%
\widetilde{\lambda }^{(c)})=\Sigma _{(s_{c-1},s_{c})}^{-1}(\widetilde{S}%
_{c}) $.\ Finally we consider the $l$-partition 
\begin{equation*}
\widetilde{{\boldsymbol{\lambda}}}=(\lambda ^{(0)},\ldots ,\widetilde{%
\lambda }^{(c)},\widetilde{\lambda }^{(c-1)}\ldots ,\lambda ^{(l-1)})
\end{equation*}
obtained by replacing in ${\boldsymbol{\lambda},}$ $\lambda ^{(c-1)}$ by $%
\widetilde{\lambda }^{(c)}$ and $\lambda ^{(c)}$ by $\widetilde{\lambda }%
^{(c-1)}$.

\begin{prop}
\label{prop_iso_sc}With the above notation, we have 
\begin{equation*}
\Psi _{e,n}^{{\mathbf{s}}\rightarrow \sigma _{c-1}({\mathbf{s}})}({%
\boldsymbol{\lambda}})=\widetilde{{\boldsymbol{\lambda}}}.
\end{equation*}
\end{prop}

\begin{proof}  The above algorithm is the
translation in terms of symbols of Proposition 5.2.2 in \cite{JL}.
\end{proof}

\begin{exmp}
Let $e=3$ and let ${\mathbf{s}}=(0,2,0)$. Then one can check that the $3$%
-partition $(4.1,3.1,1)$ belongs to $\Phi _{e}^{{\mathbf{s}}}(10)$. We want
to determine the $3$-partition $\Psi _{e,n}^{{\mathbf{s}}\rightarrow \sigma
_{2}({\mathbf{s}})}(4.1,3.1,1)$. We first write the symbol $S_{2}$ of $%
(3.1,1)$ with respect to $\mathbf{s}$ and $m=3$ : 
\begin{equation*}
S_{2}=\left( 
\begin{array}{ccccc}
0 & 1 & 3 &  &  \\ 
0 & 1 & 2 & 4 & 7
\end{array}
\right) .
\end{equation*}
By using the previous example, we obtain 
\begin{equation*}
\widetilde{S}_{2}=\left( 
\begin{array}{ccccc}
0 & 1 & 4 &  &  \\ 
0 & 1 & 2 & 3 & 7
\end{array}
\right) .
\end{equation*}
This thus gives $\widetilde{\lambda }^{(c)}=(2)$ and $\widetilde{\lambda }%
^{(c-1)}=(3).$ Hence $\Psi _{e,n}^{{\mathbf{s}}\rightarrow \sigma _{2}({%
\mathbf{s}})}(4.1,3.1.1,\emptyset )=(4.1,2,3).$
\end{exmp}

\section{Mullineux involution for Kleshchev $l$-partitions}

\subsection{Mullineux involution as a skew crystal isomorphism}

We here keep the setting of the introduction and review a result of Fayers 
\cite{Fayers} which shows how the Mullineux involution on Kleshchev $l$%
-partitions can be computed by using paths in crystal graphs. The
generalized Mullineux involution can be regarded as a bijection 
\begin{equation*}
m_{l}:\Phi _{e}^{\frak{s}}\rightarrow \Phi _{e}^{\widetilde{\mathfrak{s}}}
\end{equation*}
where $\mathfrak{s}=(s_{0}($mod $e),...,s_{l-1}($mod $e))$ and $\widetilde{%
\mathfrak{s}}=(-s_{l-1}($mod $e),...,-s_{0}($mod $e))$.

\begin{thm}[Fayers \protect\cite{Fayers}]
Let ${\boldsymbol{\lambda}}\in \Phi _{e}^{\frak{s}}$ and consider in $B_{e}^{%
\frak{s}}(\Lambda _{\frak{s}})$ a path 
\begin{equation*}
\boldsymbol{\emptyset}\overset{i_{1}}{\rightarrow }\cdot \overset{i_{2}}{%
\rightarrow }\cdot \overset{i_{3}}{\rightarrow }\cdots \overset{i_{n}}{%
\rightarrow }{\boldsymbol{\lambda}}
\end{equation*}
from the empty $l$-partition to ${\boldsymbol{\lambda}.\;}$There exists a
corresponding path 
\begin{equation*}
\boldsymbol{\emptyset}\overset{-i_{1}}{\rightarrow }\cdot \overset{-i_{2}}{%
\rightarrow }\cdot \overset{-i_{3}}{\rightarrow }\cdots \overset{-i_{n}}{%
\rightarrow }{\boldsymbol{\mu}}
\end{equation*}
in $B_{e}^{\widetilde{\mathfrak{s}}}(\Lambda _{\frak{s}})$ from the empty $l$%
-partition to an $l$-partition ${\boldsymbol{\mu}}\in \Phi _{e}^{\widetilde{%
\mathfrak{s}}}$.\ We have then 
\begin{equation*}
m_{l}({\boldsymbol{\lambda}})={\boldsymbol{\mu}.}
\end{equation*}
\end{thm}

In the level $1$ case (i.e. for $l=1$), the map $m_{1}$ can be described
independently of the notion of crystal graph by using an algorithm
originally due to Mullineux. We refer for example to \cite{KF} for a
complete exposition of this procedure. For all $s\in \mathbb{Z}$, it gives
the map 
\begin{equation*}
m_{1}:\Phi _{e}^{\frak{s}}\rightarrow \Phi _{e}^{-\frak{s}}.
\end{equation*}
where $\frak{s}=s(\text{mod }e)$ Our aim is now to show how the generalized
Mullineux involution (i.e. in level $l>1$) can be computed from the map $%
m_{1}$ and the results of Section \ref{iso}. For any $\mathfrak{s}=(s_{0}(%
\text{mod }e),...,s_{l-1}(\text{mod }e))\in (\mathbb{Z}/e\mathbb{Z})^{l}$,
set 
\begin{equation*}
\left\{ 
\begin{array}{c}
-\mathfrak{s}=(-s_{0}(\text{mod }e),...,-s_{l-1}(\text{mod }e)), \\ 
\widetilde{\mathfrak{s}}=(-s_{l-1}(\text{mod }e),...,-s_{0}(\text{mod }e)).
\end{array}
\right.
\end{equation*}
Our strategy is as follows. Let ${\boldsymbol{\lambda}}$ be in $\Phi _{e}^{%
\frak{s}}$ and consider a path from the empty $l$-partition to ${%
\boldsymbol{\lambda}}$ in $B_{e}^{\frak{s}}(\Lambda _{\frak{s}})$ with
arrows successively labelled by $i_{1},...,i_{n}.$

\begin{enumerate}
\item  We first describe the $l$-partition ${\boldsymbol{\nu}}$ in $\Phi
_{e}^{-\mathfrak{s}}$ defined by considering the path 
\begin{equation*}
\boldsymbol{\emptyset}\overset{-i_{1}}{\rightarrow }\cdot \overset{-i_{2}}{%
\rightarrow }\cdot \overset{-i_{3}}{\rightarrow }\cdots \overset{-i_{n}}{%
\rightarrow }{\boldsymbol{\nu}}
\end{equation*}
in $B_{e}^{-\frak{s}}(\Lambda _{-\frak{s}})$.

\item  Next, we describe the bijection $\Phi _{e}^{-\mathfrak{s}}\rightarrow
\Phi _{e}^{\widetilde{\mathfrak{s}}}$ defined by the crystal isomorphism
between $B_{e}^{-\frak{s}}(\Lambda _{-\frak{s}})$ and $B_{e}^{\widetilde{%
\frak{s}}}(\Lambda _{-\frak{s}}).$ This permits to compute ${\boldsymbol{\mu}%
}$ from ${\boldsymbol{\nu}.}$
\end{enumerate}

\subsection{Computing ${\boldsymbol{\nu}}$ from ${\boldsymbol{\lambda}}$}

\begin{prop}
\label{prop_mulli_composant}Let ${\boldsymbol{\lambda}}\in \Phi _{e}^{\frak{s%
}}$ and consider in $B_{e}^{\frak{s}}(\Lambda _{\frak{s}})$ a path 
\begin{equation*}
\boldsymbol{\emptyset}\overset{i_{1}}{\rightarrow }\cdot \overset{i_{2}}{%
\rightarrow }\cdot \overset{i_{3}}{\rightarrow }\cdots \overset{i_{n}}{%
\rightarrow }{\boldsymbol{\lambda}}
\end{equation*}
from the empty $l$-partition to ${\boldsymbol{\lambda}.\;}$Then there exists
a corresponding path 
\begin{equation*}
\boldsymbol{\emptyset}\overset{-i_{1}}{\rightarrow }\cdot \overset{-i_{2}}{%
\rightarrow }\cdot \overset{-i_{3}}{\rightarrow }\cdots \overset{-i_{n}}{%
\rightarrow }{\boldsymbol{\nu}}
\end{equation*}
in $B_{e}^{-\mathfrak{s}}(\Lambda _{-\frak{s}})$ from the empty $l$%
-partition to an $l$-partition ${\boldsymbol{\nu}}\in \Phi _{e}^{-%
\mathfrak{s}}$.\ We have then 
\begin{equation*}
{\boldsymbol{\nu}}=(m_{1}(\lambda ^{0}),\ldots ,m_{1}(\lambda ^{l-1})).
\end{equation*}
\end{prop}

\begin{proof}
  We know by (\ref{Fock_tens}) that the Fock spaces $%
\frak{F}_{e}^{\frak{s}}$ and $\frak{F}_{e}^{-\frak{s}}$ can be regarded as
tensor products of level $1$ Fock spaces.\ Therefore,  each $l$-partition $\mathbf{b%
}=(b^{(0)},\ldots ,b^{(l-1)})$ in the crystals $B_{e}^{\frak{s}}$ and $%
B_{e}^{-\frak{s}}$ of the Fock spaces $\frak{F}_{e}^{\frak{s}}$ and $\frak{F}%
_{e}^{-\frak{s}}$ can be written $\mathbf{b}{=b}^{(0)}{\otimes \cdots
\otimes b}^{(l-1)}{.\;}$Recall that in Kashiwara crystal basis theory, the
action of each crystal operator $\widetilde{e}_{i},\widetilde{f}_{i},$ $%
i\in \{0,\ldots ,l-1\}$ on $\mathbf{b}=b_{0}\otimes \cdots \otimes b_{l-1}$ is
determined by the intergers $\varepsilon _{i}(b_{k})=\mathrm{max}\{r\in 
\mathbb{N}\mid \widetilde{e}_{i}^{r}(b_{k})\neq 0\}$ and $\varphi
_{i}(b_{k})=\mathrm{max}\{r\mid \widetilde{f}_{i}^{r}(b_{k})\neq 0\}$ when $%
k $ runs over $\{0,\ldots ,l-1\}$.\ 

\noindent We now proceed by induction on $n$.\ For $n=0$, the proposition is
clear.\ Assume that the result  is proved for $n-1$ and set ${\boldsymbol{\lambda}=}%
\widetilde{f}_{i_{n}}\cdots \widetilde{f}_{i_{1}}(\boldsymbol{\emptyset})$
in $B_{e}^{\frak{s}}(\Lambda _{\frak{s}})$. Consider ${\boldsymbol{\lambda}}%
_{^{\flat }}{=}\widetilde{f}_{i_{n-1}}\cdots \widetilde{f}_{i_{1}}(%
\boldsymbol{\emptyset})\in B_{e}^{\frak{s}}(\Lambda _{\frak{s}})$ and ${%
\boldsymbol{\nu}}_{^{\flat }}{=}\widetilde{f}_{-i_{n-1}}\cdots \widetilde{f}%
_{-i_{1}}(\boldsymbol{\emptyset})\in B_{e}^{-\frak{s}}(\Lambda _{-\frak{s}})$%
.\ By induction, we have then ${\boldsymbol{\nu}}_{^{\flat }}=(m_{1}(\lambda
_{^{\flat }}^{(0)}),\ldots ,m_{1}(\lambda _{^{\flat }}^{(l-1)})).$ The
Mullineux involution $m_{1}$ switches the signs of the arrows.\ Thus $%
\varepsilon _{i}(m_{1}(\lambda _{^{\flat }}^{(k)}))=\varepsilon
_{-i}(\lambda _{^{\flat }}^{(k)})$ and $\varphi _{i}(m_{1}(\lambda _{^{\flat
}}^{(k)}))=\varphi _{-i}(\lambda _{^{\flat }}^{(k)})$ for any $k=0,\ldots
,l-1$.\ Since $\widetilde{f}_{i_{n}}({\boldsymbol{\lambda}}_{^{\flat }})={%
\boldsymbol{\lambda}}\neq 0$, we have $\widetilde{f}_{i_{n}}({%
\boldsymbol{\nu}}_{^{\flat }})={\boldsymbol{\nu}}\neq 0$ by the above
arguments.\ Moreover $\widetilde{f}_{i_{n}}$ and $\widetilde{f}_{-i_{n}}$
act on the same component of ${\boldsymbol{\lambda}}_{^{\flat }}$ and ${%
\boldsymbol{\nu}}_{^{\flat }}$ considered as tensor products. Namely, there
exists an integer $a\in \{0,\ldots ,l-1\}$ such that 
\begin{equation*}
\left\{ 
\begin{array}{l}
{\boldsymbol{\lambda}}=(\lambda _{^{\flat }}^{(0)},\ldots ,\widetilde{f}%
_{i_{n}}(\lambda _{^{\flat }}^{(a)}),\ldots ,\lambda _{^{\flat }}^{(l-1)})
\\ 
{\boldsymbol{\nu}}=(m_{1}(\lambda _{^{\flat }}^{(0)}),\ldots ,\widetilde{f}%
_{-i_{n}}(m_{1}(\lambda _{^{\flat }}^{(a)})),\ldots ,m_{1}(\lambda _{^{\flat
}}^{(l-1)}))
\end{array}
\right. .
\end{equation*}
We have $m_{1}(\widetilde{f}_{i_{n}}(\lambda _{^{\flat }}^{(a)}))=\widetilde{f}%
_{-i_{n}}(m_{1}(\lambda _{^{\flat }}^{(a)}))$ since  $m_{1}$ switches the sign of each arrow. 
This gives ${\boldsymbol{\nu}}%
=(m_{1}(\lambda ^{0}),\ldots ,m_{1}(\lambda ^{l-1}))$ which establishes the
proposition by induction.
\end{proof}

\subsection{Computing ${\boldsymbol{\mu}}$ from ${\boldsymbol{\nu}}${\label%
{subsec_straight}}}

\label{element}

Now, assume we have computed the Kleshchev $l$-partition ${\boldsymbol{\nu}}%
\in \Phi _{e}^{-\mathfrak{s}}$ from the Kleshchev $l$-partition ${%
\boldsymbol{\lambda}}\in \Phi _{e}^{\mathfrak{s}}$ as in the previous
section.\ To compute ${\boldsymbol{\mu}\in }\Phi _{e}^{\widetilde{%
\mathfrak{s}}}$ from the required crystal isomorphism, we need to realize ${%
\boldsymbol{\nu}}$ as an Uglov $l$-partition.\ To do this, we consider an
asymptotic multicharge $\mathbf{-s}\in -\frak{s}$.\ The multicharge $\frak{s}
$ can be written $\frak{s}=(s_{0}$ $($mod $e),\ldots ,s_{l-1}$ $($mod $e))$
with $s_{c}$ in $\{0,\ldots ,e-1\}$ for any $c=0,\ldots ,l-1$. Now for any $%
c=1,\ldots ,l-1$, write $p_{c}$ for the minimal nonnegative integer such
that 
\begin{equation*}
s_{c-1}-s_{c}+p_{c}e>n-1.
\end{equation*}
Put then : 
\begin{equation}
-{\mathbf{s}}:=(-s_{0},-s_{1}+p_{1}e,-s_{2}+p_{1}e+p_{2}e,\ldots
,-s_{l-1}+\sum_{i=1}^{l-1}p_{i}e).  \label{def-s}
\end{equation}
Since $-{\mathbf{s}}$ is asymptotic, we have ${\boldsymbol{\nu}\in }\Phi
_{e}^{-{\mathbf{s}}}$ by Prop \ref{ident} and 
\begin{equation*}
\Psi _{e,n}^{-{\mathbf{s}}\rightarrow -\mathfrak{s}}({\boldsymbol{\nu}})={%
\boldsymbol{\nu}.}
\end{equation*}
Now, with the notation of Section \ref{subsec_realization} for the
generators of the affine Weyl group of type $A$, we set 
\begin{equation*}
w_{(c,d)}=\sigma _{c}\sigma _{c-1}...\sigma _{d}\text{ for any }c\geq d>0.
\end{equation*}
Given $p\in \{1,...,l-1\}$, we also introduce 
\begin{equation*}
\gamma _{p}=w_{(l-p,1)}...w_{(l-2,p-1)}w_{(l-1,p)}.
\end{equation*}
Note that for $\mathbf{v}:=(v_{0},...,v_{l-1})\in \mathbb{Z}^{l}$ and $%
p=1,...,l-1$, we have: 
\begin{equation*}
\gamma _{p}.\mathbf{v}=(v_{p},...,v_{l-1},v_{0},...,v_{p-1}).
\end{equation*}
For each $p=1,...,l-1$, set 
\begin{equation*}
\alpha _{p}=\tau ^{l-p}\gamma _{p}.
\end{equation*}
We have thus, for $\mathbf{v}=(v_{0},...,v_{l-1})\in \mathbb{Z}^{l}$%
\begin{equation*}
\alpha _{p}.\mathbf{v}=(v_{0},...,v_{p-1},v_{p}+e,...,v_{l-1}+e).
\end{equation*}
Finally, let $w_{0}=w_{(1,1)}w_{(2,1)}...w_{(l-1,1)}$ be the longest element
of the symmetric group $S_{l}$ and put 
\begin{equation}
\eta :=\alpha _{1}^{2(p_{l-1}+1)}...\alpha _{l-1}^{2(p_{1}+1)}w_{0}.
\label{def-heta}
\end{equation}

\begin{prop}
Consider the multicharge $-{\mathbf{s}}$ of (\ref{def-s}) and set 
\begin{equation}
\widetilde{{\mathbf{s}}}=(\widetilde{s_{0}},\ldots ,\widetilde{s_{l-1}}%
):=\eta .(-{\mathbf{s}})  \label{def_sprime}
\end{equation}
Then for any $i=0,\ldots ,l-1$, we have 
\begin{equation*}
\widetilde{s_{i}}\equiv -s_{l-i-1}(\text{mod }e).
\end{equation*}
Therefore $\widetilde{\mathbf{s}}$ belongs to $\widetilde{\mathfrak{s}}$.
Moreover, the multicharge $\widetilde{\mathbf{s}}$ is asymptotic.
\end{prop}

\begin{proof}

By definition of $\widetilde{\mathbf{s}}=(\widetilde{s_{0}},...,\widetilde{s_{l-1}})$, we have for all $d=0,\ldots ,l-2$ : 
\begin{equation}
\widetilde{s_{d+1}}-\widetilde{s_{d} }=-s_{l-d-2}+s_{l-d-1}+p_{l-{d}-1}e+2e.
\label{ineqt}
\end{equation}
Now, since the integers $s_{c},$ $c=0,\ldots ,\ell -1$ belong to $\{0,\ldots
,e-1\}$, we can write 
\begin{equation*}
-s_{l-d-2}+s_{l-d-1}+2e\geq s_{l-d-2}-s_{l-d-1}.
\end{equation*}
Hence, we derive from (\ref{ineqt}) and the definition of the integers $p_{d}
$ that 
\begin{equation*}
\widetilde{s_{d+1}}-\widetilde{s_{d}}\geq s_{l-d-2}-s_{l-d-1}+p_{l-d-1}e>n-1.
\end{equation*}
Thus, the multicharge $\eta .({\mathbf{s}})$ is asymptotic which proves our
proposition.

\end{proof}

Recall that we have a Kleshchev $l$-partition ${\boldsymbol{\nu}}\in \Phi
_{e}^{-\mathfrak{s}}$. Using Section \ref{iso}, we can compute 
\begin{equation*}
\widehat{{\boldsymbol{\nu}}}:=\Psi _{e,n}^{-{\mathbf{s}}\rightarrow 
\widetilde{\mathbf{s}}}({\boldsymbol{\nu}})
\end{equation*}
where $\widetilde{\mathbf{s}}$ is defined by (\ref{def_sprime}).\ Since $%
\widetilde{\mathbf{s}}$ is asymptotic and belongs to $\widetilde{\mathfrak{s}%
},$ we derive from Proposition \ref{ident} 
\begin{equation*}
\Psi _{e,n}^{\widetilde{\mathbf{s}}\rightarrow {\widetilde{\frak{s}}}}(%
\widehat{{\boldsymbol{\nu}}})=\widehat{{\boldsymbol{\nu}}}.
\end{equation*}
Let us summarize the transformations we have computed from ${%
\boldsymbol{\lambda}}\in \Phi _{e}^{\mathfrak{s}}$. 
\begin{equation}
{\boldsymbol{\lambda}}\in \Phi _{e}^{\mathfrak{s}}\overset{\Psi _{e,n}^{{%
\frak{s}}\rightarrow -{\frak{s}}}}{\rightarrow }{\boldsymbol{\nu}}\in \Phi
_{e}^{-\mathfrak{s}}\overset{\Psi _{e,n}^{-{\frak{s}}\rightarrow -\mathbf{s}}%
}{\rightarrow }{\boldsymbol{\nu}\in }\Phi _{e}^{-\mathbf{s}}\overset{\Psi
_{e,n}^{-{\mathbf{s}}\rightarrow \widetilde{\mathbf{s}}}}{\rightarrow }%
\widehat{{\boldsymbol{\nu}}}\in \Phi _{e}^{\widetilde{\mathbf{s}}}\overset{%
\Psi _{e,n}^{-\widetilde{\mathbf{s}}\rightarrow \widetilde{\frak{s}}}}{%
\rightarrow }\widehat{{\boldsymbol{\nu}}}{\in }\Phi _{e}^{\widetilde{%
\mathfrak{s}}}  \label{funde}
\end{equation}
Recall now that we have a path 
\begin{equation*}
\boldsymbol{\emptyset}\overset{i_{1}}{\rightarrow }.\overset{i_{2}}{%
\rightarrow }.\overset{i_{3}}{\rightarrow }...\overset{i_{n}}{\rightarrow }{%
\boldsymbol{\lambda}}
\end{equation*}
in $B_{e}^{\frak{s}}$.\ By (\ref{funde}), we have the path 
\begin{equation*}
\boldsymbol{\emptyset}\overset{-i_{1}}{\rightarrow }.\overset{-i_{2}}{%
\rightarrow }.\overset{-i_{3}}{\rightarrow }...\overset{-i_{n}}{\rightarrow }%
\widehat{{\boldsymbol{\nu}}}
\end{equation*}
in $B_{e}^{\widetilde{\mathfrak{s}}}$.\ This gives 
\begin{equation*}
m_{l}({\boldsymbol{\lambda}})={\boldsymbol{\mu}=}\widehat{{\boldsymbol{\nu}}}
\end{equation*}
Hence, the procedure $(17)$ gives a purely combinatorial algorithm for
computing the Mullineux involution which does not use the crystal (and thus
the notion of good nodes) of irreducible highest weight affine modules.

\subsection{The case $e=\infty $}

We now briefly focus on the case $e=\infty $. Our algorithm considerably
simplifies because in this case, the Kleshchev and FLOTW $l$-partitions
coincide. Moreover, the Mullineux involution $m_{1}$ reduces to the ordinary
conjugation of partitions.\ Consider ${\boldsymbol{\lambda}}\in \Phi
_{\infty }^{{\mathbf{s}}}$. The computation of $m_{l}({\boldsymbol{\lambda})}
$ can be described as follows :

\begin{itemize}
\item  First, by Proposition \ref{prop_mulli_composant}, the
skew-isomorphism of crystals from $B_{e}^{\mathfrak{s}}(\Lambda _{\frak{s}})$
to $B_{e}^{-\mathfrak{s}}(\Lambda _{-\frak{s}})$ which switches the sign of
the arrows sends ${\boldsymbol{\lambda}}$ to 
\begin{equation*}
{\boldsymbol{\mu}}=((\lambda ^{0})^{\prime },...,(\lambda ^{l-1})^{\prime })
\end{equation*}
where $m_{1}$ just corresponds to the usual conjugation of partitions.

\item  Next, $w_{0}$ being the longest element of the symmetric group, we
have 
\begin{equation*}
w_{0}.(-s_{0},...,-s_{l-1})=(-s_{l-1},...,-s_{0}).
\end{equation*}
Thus we derive 
\begin{equation*}
m_{l}({\boldsymbol{\lambda}})={\boldsymbol{\nu}}:=\Psi _{n}^{{\mathbf{s}}%
\rightarrow w_{0}.{\mathbf{s}}}({\boldsymbol{\mu}}).
\end{equation*}
\end{itemize}

\subsection{Concluding remarks}

The algorithm of Section \ref{subsec-elemcrys} also permits to compute the
bijections $\Psi _{e,n}^{{\frak{s}}\rightarrow \sigma ({\frak{s}})}$ between 
$\Phi _{e}^{\mathfrak{s}}$ and $\Phi _{e}^{\sigma (\mathfrak{s})}$ for any $%
\sigma \in S_{n}$. By decomposing $\sigma $ in terms of the generators $%
s_{c},c=1,\ldots l-1$, it suffices to compute the isomorphisms $\Psi _{e,n}^{%
{\frak{s}}\rightarrow s_{c}({\frak{s}})}$.\ Since $B_{e}^{\frak{s}}$ can be
regarded as a tensor product of crystals, we can assume $l=2$ and describe
the bijection 
\begin{equation*}
\Psi _{e,n}^{({\frak{s}}_{0},{\frak{s}}_{1})\rightarrow ({\frak{s}}_{1},{%
\frak{s}}_{0})}:\Phi _{e}^{({\frak{s}}_{0},{\frak{s}}_{1})}\rightarrow \Phi
_{e}^{({\frak{s}}_{1},{\frak{s}}_{0})}.
\end{equation*}
Consider ${\boldsymbol{\lambda}\in }\Phi _{e}^{({\frak{s}}_{0},{\frak{s}}%
_{1})}$ and choose $(a_{0},a_{1})$ an asymptotic charge with $a_{0}\in {%
\frak{s}}_{0}$ and $a_{1}\in {\frak{s}}_{1}$.\ We have ${\boldsymbol{\lambda}%
\in }\Phi _{e}^{(a_{0},a_{1})}$. Then one computes $\Psi
_{e,n}^{(a_{0},a_{1})\rightarrow (a_{1},a_{0})}({\boldsymbol{\lambda}})={%
\boldsymbol{\nu}}$ as in Proposition \ref{prop_iso_sc}. Now the problem is
that $(a_{1},a_{0})$ is not asymptotic in general.\ To make it asymptotic,
we need to compose translations by $e$ in $\mathbb{Z}^{2}$ acting on the
right coordinate.\ Observe that for any $(v_{0},v_{1})\in \mathbb{Z}^{2}$,
we have $\tau s_{1}(v_{0},v_{1})=(v_{0},v_{1}+e).$ So if we set $\kappa
=\tau s_{1}$, we have to compute the bipartitions 
\begin{equation}
{\boldsymbol{\mu}}^{(k)}{=}\Psi _{e,n}^{(a_{1},a_{0})\rightarrow \tau
^{k}(a_{1},a_{0})}({\boldsymbol{\nu})}  \label{bip}
\end{equation}
with $k\in \{0,\ldots ,m\}$ where $m\in \mathbb{N}$ is minimal such that $%
a_{1}-a_{0}+me>n-1$.\ Nevertheless, in practice one can restrict our
computations and take $m$ minimal such ${\boldsymbol{\mu}}^{(m)}={%
\boldsymbol{\mu}}^{(m+1)}$. This is because the Kleshchev bipartition is
fixed by the bijection $\Psi _{e,n}^{(\frak{s}_{1},\frak{s}_{0})\rightarrow
\tau (\frak{s}_{1},\frak{s}_{0})}$. Finally, one has 
\begin{equation*}
\Psi _{e,n}^{({\frak{s}}_{0},{\frak{s}}_{1})\rightarrow ({\frak{s}}_{1},{%
\frak{s}}_{0})}({\boldsymbol{\lambda})=\boldsymbol{\mu}}^{(m)}\text{.}
\end{equation*}

Such a stabilization phenomenon also happens during the computation of ${%
\boldsymbol{\mu}}$ from ${\boldsymbol{\nu}}$ (see Section \ref
{subsec_straight}). This means that, in practice, we have 
\begin{equation*}
\Psi _{e,n}^{-{\mathbf{s}}\rightarrow \eta (-{\mathbf{s)}}}({\boldsymbol{\nu}%
})=\Psi _{e,n}^{-{\mathbf{s}}\rightarrow \eta ^{\flat }(-{\mathbf{s)}}}({%
\boldsymbol{\nu}})
\end{equation*}
where $\eta :=\alpha _{1}^{2(p_{l-1}+1)}\cdots \alpha
_{l-1}^{2(p_{1}+1)}w_{0}$ is defined as in (\ref{def-heta}) and where $\eta
^{\flat }:=\alpha _{1}^{2(p_{l-1}^{\flat }+1)}\cdots \alpha
_{l-1}^{2(p_{1}^{\flat }+1)}w_{0}$ with $p_{c}^{\flat }\leq p_{c}$ for $c\in
\{0,\ldots ,l-1\}$.

The previous description of the bijections $\Psi _{e,n}^{{\frak{s}}%
\rightarrow \sigma ({\frak{s}})}$, $\sigma \in S_{l}$ yields an alternative
for computing $\Psi _{e,n}^{-{\frak{s}}\rightarrow \widetilde{\frak{s}}}({%
\boldsymbol{\nu})=\boldsymbol{\mu}}$ in Section \ref{subsec_straight}$.\;$%
Nevertheless, one can verify that the number of elementary transformations
it requires (i.e. the number of bijections $\Psi _{e,n}^{\mathbf{s}%
\rightarrow s_{c}(\mathbf{s})},$ $c=1,\ldots ,l-1$ or $\Psi _{e,n}^{\mathbf{s%
}\rightarrow \tau (\mathbf{s})}$ to compute) is in general greater than the
number of transformations used in $\Psi _{e,n}^{-{\mathbf{s}}\rightarrow
\eta (-{\mathbf{s)}}}$.

\begin{exmp}
We conclude this paper with a computation of the Mullineux involution on a $%
3 $-partition. Take $e=4$ and $\frak{s}=(0(\text{mod }4),1(\text{mod }4),3(%
\text{mod }4))$. Then, one can check that the $3$-partition ${%
\boldsymbol{\lambda}}:=(4,3,1.1.1))$ is a Kleshchev $3$-partition of rank $%
n=10$ associated to $\frak{s}$. Now we have 
\begin{equation*}
m_{1}(4)=(2.1.1),\ m_{1}(3)=(1.1.1),\ m_{1}(1.1.1)=(3).
\end{equation*}
Hence, to ${\boldsymbol{\lambda}}$, one can attach the Kleshchev $3$%
-partition ${\boldsymbol{\nu}}=(2.1.1,1.1.1,3)$ which is Kleshchev for $-%
\frak{s}=(0(\text{mod }4),-1(\text{mod }4),-3(\text{mod }4))$. Since $%
p_{1}=p_{2}=3$, This $3$-partition is an Uglov $3$-partition associated with
the multicharge 
\begin{equation*}
-\mathbf{s}=(0,-1+3\times 4,-3+3\times 4+3\times 4)=(0,11,21).
\end{equation*}
Following \S \ref{element}, we put : 
\begin{equation*}
\eta =\alpha _{1}^{8}\alpha _{2}^{8}w_{0}.
\end{equation*}
We compute 
\begin{equation*}
\Psi _{e,n}^{-\mathbf{s}\rightarrow \eta .(-\mathbf{s})}({\boldsymbol{\nu}}%
)=(\emptyset ,1.1.1,4.1.1.1).
\end{equation*}
Hence, we obtain 
\begin{equation*}
m_{3}({\boldsymbol{\lambda}})=(\emptyset ,1.1.1,4.1.1.1).
\end{equation*}
\end{exmp}

\noindent \textbf{Acknowledgments.} The authors are grateful to A. Kleshchev
for useful comments and for having pointed out several references. Part of
this work has been written while the authors were visiting the Mathematical
Research Institute of Mathematics in Berkeley for the programs
``Combinatorial Representation Theory'' and ``Representation Theory of
Finite Groups and Related Topics''. The authors want to thank the organizers
of these programs and the MSRI for financial supports.

\end{document}